\documentclass{amsart}
\usepackage{mathrsfs,latexsym,amsfonts,amssymb}
\usepackage{hyperref}

\newtheorem{theorem}{Theorem}[section]
\newtheorem{lemma}[theorem]{Lemma}
\newtheorem{corollary}[theorem]{Corollary}
\newtheorem{question}[theorem]{Question}
\theoremstyle{definition}
\newtheorem{definition}[theorem]{Definition}

\theoremstyle{remark}

\newtheorem{example}[theorem]{Example}

\begin{document}
\title[Uniform covers at non-isolated points]%
{Uniform covers at non-isolated points}

\author{Fucai Lin}
\address{Fucai Lin: Department of Mathematics,
Zhangzhou Normal University, Zhangzhou 363000, P. R. China}
\email{lfc19791001@163.com}
\author{Shou Lin}
\address{Shou Lin (the corresponding author): Department of Mathematics,
Zhangzhou Normal University, Zhangzhou 363000, P. R. China;
Institute of Mathematics, Ningde Teachers' College, Ningde, Fujian
352100, P. R. China}
\email{linshou@public.ndptt.fj.cn}

\thanks{Supported in part by the NSFC(No. 10571151).}


\begin{abstract}
In this paper,\ the authors define a space with an uniform base
at non-isolated points, give some characterizations of images of
metric spaces by boundary-compact maps, and study certain
relationship among spaces with special base properties.\ The main
results are the following: (1)\ $X$ is an open,\ boundary-compact
image of a metric space if and only if $X$ has an uniform base at
non-isolated points; (2)\ Each discretizable space of a space with
an uniform base is an open compact and at most boundary-one image of
a space with an uniform base; (3)\ $X$ has a point-countable base if
and only if $X$ is a bi-quotient,\ at most boundary-one and
countable-to-one image of a metric space.
\end{abstract}

\maketitle

\section{Introduction}
Topologists obtained many interesting characterizations of the
images of metric spaces by some kind of maps.\ A. V.
Arhangel'ski\v{\i} \cite{Ar1} proved that a space $X$ is an open
compact image of a metric space if and only if $X$ has an uniform
base.\ Recently, C. Liu \cite{LC} gives a new characterization of
spaces with a point-countable base by pseudo-open and at most
boundary-one images of metric spaces. How to character an open or
pseudo-open and boundary-compact images of metric spaces?\ On the
other hand, a study of spaces with a sharp base or a weakly uniform
base \cite{AJRS, Au} shows that some properties of a non-isolated
point set of a topological space will help us discuss a whole
construction of a space. In this paper,\ the authors analyze some
base properties on non-isolated points of a space, introduce a space
having an uniform base at non-isolated points and describe it as an
image of a metric space by open boundary-compact maps.\ Some
relationship among the images of metric spaces under open
boundary-compact maps, pseudo-open boundary-compact maps, open
compact maps, and spaces with a point-countable base are discussed.

By $\mathbb{R, N}$, denote the set of real numbers and positive
integers, respectively. For a space $X$, let $$I(X)=\{x:x \mbox{ is
an isolated point of } X\}$$ and $$\mathcal I(X)=\{\{x\}:x\in
I(X)\}.$$

In this paper all spaces are $T_{2}$, all maps are continuous and
onto. Recalled some basic definitions.

Let $X$ be a topological space. $X$ is called a {\it metacompact}
(resp. {\it paracompact}, {\it meta-Lindel\"{o}f}) space if every
open cover of $X$ has a point-finite (resp. locally finite,
point-countable) open refinement.\ $X$ is said to have a
$G_{\delta}$-{\it diagonal} if the diagonal $\Delta =\{(x,x):x\in
X\}$ is a $G_\delta$-set in $X\times X$. $X$ is called a {\it
perfect space} if every open subset of $X$ is an $F_{\sigma}$-set in
$X$.

\begin{definition}
Let $\mathcal{P}$ be a base of a space $X$.
\begin{enumerate}
\item $\mathcal{P}$ is an {\it uniform base}~\cite{AP}~(resp.
{\it uniform base at non-isolated points}) for $X$ if for each
($resp.$\ non-isolated) point $x\in X$ and $\mathcal{P}^{\prime}$ is
a countably infinite subset of $(\mathcal{P})_{x}$,
$\mathcal{P}^{\prime}$ is a neighborhood base at $x$.

\item $\mathcal{P}$ is a {\it point-regular base}~\cite{AP}~(resp.
{\it point-regular base at non-isolated points}) for $X$ if for each
($resp.$\ non-isolated) point $x\in X$ and $x\in U$ with $U$ open in
$X$, $\{P\in (\mathcal{P})_{x}:P\not\subset U\}$ is finite.
\end{enumerate}
\end{definition}

In the definition, ``at non-isolated points'' means ``at each
non-isolated point of $X$''. It is obvious that uniform bases~(resp.
point-regular bases)$\Rightarrow$ uniform bases at non-isolated
points (resp. point-regular bases at non-isolated points),\ but we
will see that uniform bases at non-isolated points~(resp.
point-regular bases at non-isolated points) $\not\Rightarrow$
uniform bases~(resp. point-regular bases) by Example 4.1.

\begin{definition}
Let $X$ be a space, and $\{\mathcal{P}_{n}\}$ a sequence of open
subsets of $X$.
\begin{enumerate}
\item $\{\mathcal{P}_{n}\}$ is called a {\it quasi-development}~\cite{Be} for $X$
if for every $x\in U$ with $U$ open in $X$, there exists $n\in
\mathbb{N}$ such that $x\in\mbox{st}(x,\mathcal{P}_{n})\subset U$.

\item $\{\mathcal{P}_{n}\}$ is called a {\it
development}~(resp. {\it development at non-isolated points}) for
$X$ if $\{\mbox{st}(x,\mathcal{P}_{n})\}_{n\in\mathbb N}$ is a
neighborhood base at $x$ in $X$ for each~(resp. non-isolated) point
$x\in X$.

\item $X$ is called {\it quasi-developable}~(resp. {\it developable},
{\it developable at non-isolated points}) if $X$ has a
quasi-development (resp. development, development at non-isolated
points).
\end{enumerate}
\end{definition}

It is obvious that every development for a space is a development at
non-isolated points,\ but a space having a development at
non-isolated points may not have a development, see Example 4.2.

\begin{definition}
Let $f:X\rightarrow Y$ be a map.
\begin{enumerate}
\item $f$ is a {\it compact map}~(resp. {\it s-map}) if each
$f^{-1}(y)$ is compact~(resp. separable) in $X$;

\item $f$ is a {\it boundary-compact map}~(resp. {\it
boundary-finite map}, {\it at most boundary-one map}) if each
$\partial f^{-1}(y)$ is compact~(resp. finite, at most one point) in
$X$;

\item $f$ is an {\it open map} if whenever
$U$ open in $X$, then $f(U)$ is open in $Y$;

\item $f$ is a {\it bi-quotient map}~(resp. {\it countably bi-quotient map})
if for any $y\in Y$ and any (resp. countable) family $\mathcal{U}$
of open subsets in $X$ with $f^{-1}(y)\subset \cup \mathcal{U}$,\
there exists finite subset $\mathcal{U}^{\prime}\subset \mathcal{U}$
such that $y\in \mbox{Int}f(\cup \mathcal{U}^{\prime})$;

\item $f$ is a {\it pseudo-open map} if whenever
$f^{-1}(y)\subset U$ with $U$ open in $X$, then
$y\in\mbox{Int}(f(U))$.
\end{enumerate}
\end{definition}

It is easy to see that open maps~$\Rightarrow$ bi-quotient
maps~$\Rightarrow$ countably bi-quotient maps~$\Rightarrow$
pseudo-open maps~$\Rightarrow$ quotient maps.

\begin{definition}
Let $X$ be a space.
\begin{enumerate}
\item A collection $\mathcal{U}$ of subsets of $X$ is said to be $Q$ (i.e., {\it
interior-preserving}) if $\mbox{Int}(\cap \mathcal{W})=\cap
\{\mbox{Int}W:W\in \mathcal{W}\}$ for every $\mathcal{W}\subset
\mathcal{U}$.

\item An {\it ortho-base}~\cite{LN} $\mathcal{B}$ for $X$ is a base of $X$ such
that either $\cap \mathcal{A}$ is open in $X$ or $\cap
\mathcal{A}=\{x\}\notin\mathcal I(X)$ and $\mathcal{A}$ is a
neighborhood base at $x$ in $X$ for each $\mathcal{A}\subset
\mathcal{B}$. A space $X$ is a {\it proto-metrizable space}
\cite{GZ} if it is a paracompact space with an ortho-base.

\item A {\it sharp base} \cite{AAC}~$\mathcal{B}$ of $X$ is a base of $X$
such that, for every injective sequence $\{B_{n}\}\subset
\mathcal{B}$, if $x\in \bigcap_{n\in\mathbb N}B_{n}$, then
$\{\bigcap_{i\leq n}B_{i}\}_{n\in\mathbb N}$ is a neighborhood base
at $x$.

\item A base $\mathcal{B}$ of $X$ is said to be $BCO$~(i.e., bases of countable orders)
if, for any $x\in X$, $\{B_{i}\}\subset \mathcal{B}$ is a strictly
decreasing sequence, then $\{B_{i}\}_{i\in\mathbb N}$ is a
neighborhood base at $x$.
\end{enumerate}
\end{definition}

It is well known that \cite{AAC, AJRS, Au}
\begin{enumerate}
\item Uniform bases~$\Rightarrow\sigma$-point-finite
bases~$\Rightarrow\sigma$-Q bases;

\item Uniform bases~$\Rightarrow$ sharp bases, developable
spaces~$\Rightarrow$ BCO, $G_{\delta}$-diagonals;

\item Sharp bases~$\Rightarrow$ point-countable bases.
\end{enumerate}

Readers may refer to \cite{En, Ls} for unstated definitions and
terminology.

\vskip 0.5cm

\section{Some lemmas}
In this section some technical lemmas are given.

\begin{lemma}
Let $\mathcal{P}$ be a base for a space $X$.\ Then the following are
equivalent:
\begin{enumerate}
\item $\mathcal{P}$ is an uniform base at non-isolated points for $X$;

\item $\mathcal{P}$ is a point-regular base at non-isolated points for $X$.
\end{enumerate}
\end{lemma}

\begin{proof}
$(2)\Rightarrow (1)$ is trivial.\ We only need to prove
$(1)\Rightarrow (2)$.

Let $\mathcal{P}$ be an uniform base at non-isolated points for
$X$.\ If there exist a non-isolated point $x\in X$ and an open
subset $U$ in $X$ with $x\in U$ such that $\{P\in
(\mathcal{P})_x:P\not\subset U\}$ is infinite.\ Take
$\{P_{n}:n\in\mathbb N\}\subset\{P\in (\mathcal{P})_x:P\not\subset
U\}$, and choose $x_{n}\in P_{n}\setminus U$ for each $n\in\mathbb
N$.\ Then $\{P_n\}_{n\in\mathbb N}$ is a neighborhood base at $x$,
thus the sequence $\{x_{n}\}$ converges to $x$ in $X$.\ Hence
$x_{m}\in U$ for some $m\in \mathbb{N}$, a contradiction.\
Therefore, $\mathcal{P}$ is a point-regular base at non-isolated
points for $X$.
\end{proof}

\begin{lemma}
Let $\{\mathcal{P}_{n}\}$ be a development at non-isolated points
for a space $X$.\ If $\mathcal{P}_{n}$ is point-finite at each
non-isolated point and $\mathcal{P}_{n+1}$ refines $\mathcal{P}_{n}$
for each $n\in \mathbb{N}$,\ then $\mathcal{P}=\mathcal
I(X)\cup(\bigcup_{n\in \mathbb{N}}\mathcal{P}_{n})$ is an uniform
base at non-isolated points for $X$.
\end{lemma}

\begin{proof}
Let $x$ be a non-isolated point in $X$, and $\{P_{m}: m\in\mathbb
N\}$ an infinite subset of $(\mathcal{P})_{x}$. By the
point-finiteness, there exists $P_{m_{k}}\in \mathcal{P}_{n_{k}}$
such that $m_{k}<m_{k+1},$ and $n_{k}<n_{k+1}$ for each $k\in
\mathbb{N}$.\ Since $\{\mathcal{P}_{n}\}$ is a development at
non-isolated points for $X$, $\{P_{m_{k}}\}_{k\in\mathbb N}$ is a
neighborhood base at $x$ in $X$, so $\{P_{m}\}_{m\in\mathbb N}$ is a
neighborhood base at $x$.\ Thus $\mathcal{P}$ is a uniform base at
non-isolated points for $X$.
\end{proof}

Let $\mathcal P$ be a family of subsets of a space $X$.\
$\mathcal{P}$ is called {\it point-finite at non-isolated points}
(resp. {\it point-countable at non-isolated points}) if for each
non-isolated point $x\in X$, $x$ is belong to at most finite(resp.
countable) elements of $\mathcal{P}$.\ Let $\{\mathcal P_n\}$ be a
development (resp. a development at non-isolated points) for $X$.
$\{\mathcal P_n\}$ is said to be a {\it point-finite development}
(resp. {\it a point-finite development at non-isolated points}) for
$X$ if each $\mathcal P_n$ is point-finite at each (resp.
non-isolated) point of $X$.

\begin{lemma}
A space $X$ has an uniform base at non-isolated points if and only
if $X$ has a point-finite development at non-isolated points.
\end{lemma}

\begin{proof}
Sufficiency. It is easy to see by Lemma 2.2.

Necessity.\ Let $\mathcal{P}$ be an uniform base at non-isolated
points for $X$.\ Then $\mathcal{P}$ is a point-regular base at
non-isolated points by Lemma 2.1. We can assume that if
$P\in\mathcal P$ and $P\subset I(X)$, $P$ is a single point set.

Claim: Let $x$ be a non-isolated point of $X$ and $y\neq x$. Then
$\{H\in\mathcal{P}:\{x, y\}\subset H\}$ is finite.

In fact, $\{H\in\mathcal{P}:\{x, y\}\subset H\}\subset
(\mathcal{P})_{x}$. If $\{H\in\mathcal{P}:\{x, y\}\subset H\}$ is
infinite, then it is a local base at $x$, hence $y\rightarrow x$,
this is a contradiction.

(a)\ $\mathcal{P}$ is point-countable at non-isolated points in $X$.

Let $x\in X$ be a non-isolated point, there is a non-trivial
sequence $\{x_{n}\}$ converging to $x$. By the Claim,
$\{P\in(\mathcal{P})_{x}: x_{n}\in P\}$ is finite for each $n$, then
$(\mathcal{P})_{x}=\bigcup_{n\in\mathbb{N}}\{P\in(\mathcal{P})_{x}:
x_{n}\in P\}$ is countable.

A family $\mathcal F$ of subsets of $X$ is called having the
property $(\sharp)$\ if for any $F\in \mathcal{F}\setminus\mathcal
I(X)$,\ then $\{H\in \mathcal{F}:F\subset H\}$ is finite.

(b)\ $\mathcal{P}$ has the property ($\sharp$).

Since $F\in\mathcal{P}\setminus\mathcal{I}(X)$. Then $F$ contains a
non-isolated point and $|F|>1$. By the Claim, $\mathcal{P}$ has the
property $(\sharp)$.

Put

\hspace{1cm}$\mathcal{P}^{m}=\{H\in \mathcal{P}:\mbox{if~}H\subset
P\in \mathcal{P}, \mbox{then~}P=H\}\cup\mathcal I(X)$,

\hspace{1cm}$\mathcal{P}^{\prime}=(\mathcal{P}\setminus
\mathcal{P}^{m})\cup \mathcal I(X).$

(c)\ $\mathcal{P}^{m}$ is an open cover, and is point-finite at
non-isolated points for $X$.

There exists $H_P\in \mathcal{P}^{m}$ such that $P\subset H_P$ for
each $P\in \mathcal{P}\setminus \mathcal I(X)$ by (b).\ Thus
$\mathcal{P}^{m}$ is an open cover of $X$.\ If $\mathcal{P}^{m}$ is
not point-finite at some non-isolated point $x\in X$,\ then there
exists an infinite subset $\{H_{n}:n\in\mathbb N\}$ of
$(\mathcal{P}^{m})_{x}$.\ For each $n\in\mathbb N$,
$H_{n}\not\subset H_1$, there exists $x_{n}\in H_{n+1}\setminus
H_{1}$.\ Then $x_n\to x\in H_1$, a contradiction.

(d)\ $\mathcal{P}^{\prime}$ is a point-regular base at non-isolated
points for $X$.

Let $x\in U\setminus I(X)$ with $U$ open in $X$.\ There exist
$V,W\in \mathcal{P}$ and $y\in V\setminus \{x\}$ such that $x\in
W\subset V\setminus \{y\}\subset V\subset U$.\ Thus $W\in
\mathcal{P}^{\prime}$.\ Then $\mathcal{P}^{\prime}$ is a base for
$X$,\ and it is a point-regular base at non-isolated points for $X$.

Put $\mathcal{P}_{1}=\mathcal{P}^{m}$ and
$\mathcal{P}_{n+1}=[(\mathcal{P}\setminus \bigcup_{i\leq
n}\mathcal{P}_{i})\bigcup \mathcal I(X)]^{m}$ for any $n\in
\mathbb{N}$. Then $\mathcal P=\bigcup_{n\in\mathbb
N}\mathcal{P}_{n}$ by (b).

(e)\ $\{\mathcal P_n\}$ is a point-finite development at
non-isolated points for $X$.

Each $\mathcal P_n$ is point-finite at non-isolated points by (c)
and (d). If $x\in U\setminus I(X)$ with $U$ open in $X$, then
$\{P\in(\mathcal P)_x:P\not\subset U\}$ is finite, thus there is
$n\in\mathbb N$ such that $P\subset U$ whenever $x\in P\in\mathcal
P_n$, i.e., $\mbox{st}(x, \mathcal P_n)\subset U$.\ So
$\{\mathcal{P}_{n}\}$ is a development at non-isolated points.
\end{proof}

\begin{lemma}\cite{Ar1, Ar2, He}
The following are equivalent for a space $X$:
\begin{enumerate}

\item $X$ is an open compact image of a metric space;

\item $X$ is an pseudo-open compact image of a metric space;

\item $X$ has an uniform base;

\item $X$ has a point-regular base;

\item $X$ is a metacompact and developable space;

\item $X$ is a space with a point-finite development.
\end{enumerate}
\end{lemma}

\begin{lemma}
Each pseudo-open,\ boundary-compact map is a bi-quotient map.
\end{lemma}

\begin{proof}
Let $f:X\rightarrow Y$ be a pseudo-open,\ boundary-compact map.\ For
each $y\in Y$ and a family $\mathcal{U}$ of open subsets in $X$ with
$f^{-1}(y)\subset \cup \mathcal{U}$,\ $\partial f^{-1}(y)\subset
\cup \mathcal{U}^{\prime}$ for some finite
$\mathcal{U}^{\prime}\subset\mathcal U$.\ We can assume that there
exists $U\in \mathcal{U}^{\prime}$ such that $U\cap f^{-1}(y)\neq
\emptyset$.\ Thus $y\in f(U)$.\ Let $V=(\cup
\mathcal{U}^{\prime})\cup\mbox{Int}(f^{-1}(y))$.\ Then
$f^{-1}(y)\subset V$.\ Since $f$ is pseudo-open, thus $$y\in
\mbox{Int}(f(V))\subset f((\cup \mathcal{U}^{\prime})\cup
f^{-1}(y))= f(\cup \mathcal{U}^{\prime})\cup \{y\}=f(\cup
\mathcal{U}^{\prime}),$$ so $f(\cup \mathcal{U}^{\prime})$ is a
neighborhood of $y$ in $Y$.\ Hence $f$ is a bi-quotient map.
\end{proof}

\vskip 0.5cm

\section{Main Results}
In this section spaces with an uniform base at non-isolated points
are discussed, and some characterizations of images of metric spaces
by boundary-compact maps are given.

\begin{theorem}
The following are equivalent for a space $X$:
\begin{enumerate}
\item $X$ is an open,\ boundary-compact image of a metric space;

\item $X$ has an uniform base at non-isolated points;

\item $X$ has a point-regular base at non-isolated points;

\item $X$ has a point-finite development at non-isolated points.
\end{enumerate}
\end{theorem}

\begin{proof}
It is obvious that $(2)\Leftrightarrow (3)\Leftrightarrow (4)$ by
Lemma 2.1 and Lemma 2.3.

$(1)\Rightarrow (4)$.\ Let $M$ be a metric space and $f:M\rightarrow
X$ an open,\ boundary-compact map.\ By \cite[5.4.E]{En}, we can
choose a sequence $\{\mathcal{B}_{i}\}$ of open covers of $M$ such
that $\{\mbox{st}(K,\mathcal{B}_{i})\}_{i\in\mathbb N}$ is a
neighborhood base of $K$ in $M$ for each compact subset $K\subset
M$.\ For each $i\in\mathbb N$, we can assume that
$\mathcal{B}_{i+1}$ is a locally finite open refinement of $\mathcal
B_i$, and set $\mathcal{P}_{i}=f(\mathcal{B}_{i})\cup \mathcal
I(X)$.\ Then $\mathcal P_i$ is an open cover of $X$ for each $i\in
\mathbb{N}$. If $x$ is an accumulation point of $X$, then
$\mbox{Int}f^{-1}(x)=\emptyset$, thus $f^{-1}(y)=\partial f^{-1}(x)$
is compact in $M$, hence $\{B\in\mathcal B_i:B\cap
f^{-1}(x)\not=\emptyset\}$ is finite by the local finiteness of
$\mathcal B_i$, i.e., $(\mathcal P_i)_x$ is finite.\ This shows that
$\mathcal{P}_{i}$ is point-finite at non-isolated points.\ Next, we
will prove that $\{\mathcal{P}_i\}$ is a development at non-isolated
points for $X$.\ Let $x\in U\setminus I(X)$ with $U$ open in $X$.\
Since $f^{-1}(x)$ is compact,\ there exists $m\in \mathbb{N}$ such
that $\mbox{st}(f^{-1}(x),\mathcal{B}_{m})\subset f^{-1}(U)$,\ so
$\mbox{st}(x,\mathcal{P}_{m})=\mbox{st}(x,f(\mathcal{B}_{m}))\subset
U$.\ Thus $\{\mathcal{P}_i\}$ is a point-finite development at
non-isolated points for $X$.

$(4)\Rightarrow (1)$.\ First, a metric space $M$ and a function
$f:M\to X$ are defined as follows. Let $\{\mathcal{P}_{n}\}$ be a
point-finite development at non-isolated points for $X$.\ For each
$n\in\mathbb N$, assume that $\mathcal I(X)\subset\mathcal P_n$, put
$\mathcal{P}_{n}=\{P_{\alpha}:\alpha \in \Lambda_{n}\}$ and endow
$\Lambda_{n}$ a discrete topology.\ Put
$$M=\{\alpha=(\alpha_{n})\in
\prod_{n\in \mathbb{N}}\Lambda_{n}: \{P_{\alpha_{n}}\}_{n\in\mathbb
N}\mbox{ is a neighborhood base}$$ $$\mbox{at some }x_{\alpha}\in
X\}.$$ Then $M$,\ which is a subspace of the product space
$\prod_{n\in \mathbb{N}}\Lambda_{n}$,\ is a metric space.\ Define a
function $f:M\rightarrow X$ by $f((\alpha_{n}))=x_{\alpha}$.\ Then
$f((\alpha_{n}))=\bigcap_{n\in\mathbb N}P_{\alpha_n}$, and $f$ is
well defined.\ $(f,M,X,\mathcal{P}_{n})$ is called a Ponomarev
system. It is easy to see that $f$ is a map. The following will
prove that $f$ is an open boundary-compact map.

(a)\ $f$ is an open map.

For any $\alpha =(\alpha_{n})\in M,n\in \mathbb{N}$,\ put
$$B(\alpha_{1},\alpha_{2},\cdots ,\alpha_{n})=\{(\beta_{i})\in
M:\beta_{i}=\alpha_{i}\ \mbox{whenever $i\leq n$}\}.$$ Then
$f(B(\alpha_{1},\alpha_{2},\cdots ,\alpha_{n}))=\bigcap_{i\leq
n}P_{\alpha_{i}}$.\ In fact,\ if $\beta =(\beta_{i})\in
B(\alpha_{1},\alpha_{2},\cdots ,\alpha_{n})$,
$f(\beta)=\bigcap_{i\in \mathbb{N}}P_{\beta_{i}}\subset
\bigcap_{i\leq n}P_{\alpha_{i}}$.\ Thus
$$f(B(\alpha_{1},\alpha_{2},\cdots ,\alpha_{n}))\subset
\bigcap_{i\leq n}P_{\alpha_{i}}.$$ On the other hand,\ let $x\in
\bigcap_{i\leq n}P_{\alpha_{i}}$.\ Choose a countable family
$\{P_{\beta_{i}}\}_{i\in \mathbb{N}}$ of subsets of $X$ such that

(i)\ $x\in P_{\beta_{i}}\in\mathcal P_i$ for each $i\in\mathbb N$,

(ii)\ $\beta_{i}=\alpha_{i}$ whenever $i\leq n$, and

(iii)\ $P_{\beta_i}=\{x\}$ whenever $i>n$ and $x\in I(X)$.

Put $\beta =(\beta_{i})$.\ Then $\beta\in
B(\alpha_{1},\alpha_{2},\cdots ,\alpha_{n})$, and $f(\beta)=x$.\
Thus $\bigcap_{i\leq n}P_{\alpha_{i}}\subset
f(B(\alpha_{1},\alpha_{2},\cdots ,\alpha_{n}))$.

In conclusion,\ $f(B(\alpha_{1},\alpha_{2},\cdots
,\alpha_{n}))=\bigcap_{i\leq n}P_{\alpha_{i}}$.\ Since
$$\{B(\alpha_{1},\alpha_{2},\cdots ,\alpha_{n}):(\alpha_{i})\in M,
n\in \mathbb{N}\}$$ is a base of $M$, $f$ is an open map.

(b)\ $f$ is a boundary-compact map.

Let $x\in X$.\ If $x\in I(X)$,\ then $\partial
f^{-1}(x)=\emptyset$.\ If $x\not\in I(X)$,\ $\partial
f^{-1}(x)=f^{-1}(x)$ by (b).\ For each $i\in\mathbb N$, let
$\Gamma_{i}=\{\alpha\in \Lambda_{i}:x\in P_{\alpha}\}$.\ Then
$\Gamma_{i}$ is finite.\ Thus $\prod_{i\in \mathbb{N}}\Gamma_{i}$ is
a compact subset of $\prod_{i\in \mathbb{N}}\Lambda_{i}$.\ We only
need to proof $f^{-1}(x)=\prod_{i\in \mathbb{N}}\Gamma_{i}$.\
Indeed,\ if $\alpha =(\alpha_{i})\in \prod_{i\in
\mathbb{N}}\Gamma_{i}$,\ then $\{P_{\alpha_{i}}\}_{i\in\mathbb N}$
is a neighborhood base at $x$ for $X$.\ Thus $\alpha\in M$ and
$f(\alpha)=x$,\ so $\prod_{i\in \mathbb{N}}\Gamma_{i}\subset
f^{-1}(x)$.\ On the other hand,\ if $\alpha =(\alpha_{i})\in
f^{-1}(x)$,\ then $x\in \bigcap_{i\in \mathbb{N}}P_{\alpha_{i}}$ and
$\alpha\in \prod_{i\in \mathbb{N}}\Gamma_{i}$.\ So $f^{-1}(x)\subset
\prod_{i\in \mathbb{N}}\Gamma_{i}$.\ Thus $\partial
f^{-1}(x)=f^{-1}(x)=\prod_{i\in \mathbb{N}}\Gamma_{i}$ is compact.
\end{proof}

In the Ponomarev system $(f, M, X, \mathcal P_n)$,\ it is always
hold that $f^{-1}(x)\subset \prod_{i\in \mathbb{N}}\{\alpha\in
\Lambda_{i}:x\in P_{\alpha}\}$ for each $x\in X$. The following
corollary is obtained.

\begin{corollary}
A space $X$ has a point-countable base which is uniform at
non-isolated points if and only if $X$ is an open boundary-compact,
$s$-image of a metric space.
\end{corollary}

\begin{corollary}
Each space having an uniform base at non-isolated points is
preserved by an open, boundary-finite map.
\end{corollary}

\begin{proof}
let $f:X\rightarrow Y$ be an open boundary-finite map, where $X$ has
an uniform base at non-isolated points. There exist a metric space
$M$ and an open boundary-compact map $g:M\rightarrow X$ by Theorem
3.1. Since $\partial (f\circ g)^{-1}(y)\subset \bigcup\{\partial
g^{-1}(x):x\in \partial f^{-1}(y)\}$ for each $y\in Y$,\ $f\circ
g:M\rightarrow Y$ is an open boundary-compact map. Hence $Y$ has an
uniform base at non-isolated points.
\end{proof}

\begin{theorem}
Let $X$ be a space having an uniform base at non-isolated points.
Then
\begin{enumerate}
\item $X$ is a quasi-developable space;

\item $X$ has an ortho-base and a $\sigma$-Q base.
\end{enumerate}
\end{theorem}

\begin{proof}
By Theorem 3.1, let $\{\mathcal P_n\}_{n\in\mathbb N}$ be a
point-finite development at non-isolated points for $X$. Put
$\mathcal P_0=\mathcal I(X)$.\ It is easy to check that $\{\mathcal
P_n\}_{n\in\omega}$ is a quasi-development for $X$.

Let $\mathcal{P}=\bigcup_{n\in \omega}\mathcal{P}_{n}$. Then
$\mathcal{P}$ is a $\sigma$-Q base and an ortho-base for $X$.

First, $\mathcal{P}_{n}$ is interior-preserving for each
$n\in\mathbb N$. Indeed, for each $\mathcal{A}\subset
\mathcal{P}_{n}$, if $x\in \cap \mathcal{A}-I(X)$, then
$(\mathcal{P}_{n})_{x}$ is finite, thus $\cap \mathcal{A}$ is a
neighborhood of $x$ in $X$. So $\mathcal{P}$ is a $\sigma$-Q base
for $X$.

Secondly, let $\mathcal{A}\subset \mathcal{P}$ with $\cap
\mathcal{A}$ not open in $X$. Then there exists $x\in \cap
\mathcal{A}$ such that $\cap \mathcal{A}$ is not a neighborhood of
$x$ in $X$, thus $x$ is a non-isolated point and
$(\mathcal{P}_{n})_{x}$ is finite for each $n\in\mathbb N$. Let
$x\in U$ with $U$ open in $X$. There exists $n\in \mathbb{N}$ such
that $x\in\mbox{st}(x,\mathcal{P}_{n})\subset U$. Choose $m\geq n$
and $A\in \mathcal{A}\cap \mathcal{P}_{m}$. Then $A\subset
\mbox{st}(x,\mathcal{P}_{n})\subset U$, thus $\mathcal{A}$ is a
neighborhood base at $x$ in $X$. So $\cap \mathcal{A}$ is a single
point subset. Hence $\mathcal{P}$ is an ortho-base for $X$.
\end{proof}

\begin{corollary}
Let $X$ be a space having an uniform base at non-isolated points.
Then $(1)\Rightarrow (2)\Leftrightarrow (3)$ in the following:
\begin{enumerate}
\item $X$ has a sharp base;

\item $X$ is a developable space;

\item $I(X)$ is an $F_{\sigma}$-set in $X$.
\end{enumerate}
\end{corollary}

\begin{proof}
$(1)\Rightarrow (3)$ is proved in \cite[Theorem 3.1]{BB} for any
space $X$.\ $(2)\Rightarrow (3)$ is obvious because each open subset
of a developable space is an $F_{\sigma}$-set.

$(3)\Rightarrow (2)$. Let $\{\mathcal{B}_{n}\}$ be a point-finite
development at non-isolated points for $X$ by Theorem 3.1.\ Since
$I(X)$ is an $F_{\sigma}$-set, there exists a sequence $\{G_{n}\}$
of open subsets of $X$ such that $X-I(X)=\bigcap_{n\in
\mathbb{N}}G_{n}$. For each $n\in \mathbb{N}$, let
$$\mathcal{U}_{n}=\{G_{n}\}\cup \{\{x\};x\in X-G_{n}\}.$$
Then $\{\mathcal{B}_{n}, \mathcal{U}_{n}\}$ is a development for
$X$. Hence $X$ is a developable space.
\end{proof}

The following corollary is hold by Lemma 2.4.

\begin{corollary}
A space $X$ is an open compact image of a metric space if and only
if $X$ is a perfect, metacompact space, which is an open
boundary-compact image of a metric space.
\end{corollary}

By the corollary, some metrizable theorems on spaces with an uniform
base at non-isolated points can be obtained.\ For example, let $X$
be a space with an uniform base at non-isolated points, then $X$ is
metrizable if and only if it is a perfect, collectionwise normal
space.

Now, a special space with an uniform base at non-isolated points is
discussed.\ Let $(X, \tau)$ be a space and $A\subset X$.\ $X$ is
said to be discretizable by $A$ if $X$ is endowed with the topology
generated by $\tau\cup\{\{x\}:x\in A\}$ as a base for $X$ \cite{LN}.
Denote the discretizable space of $X$ by $X_A$.

It is obvious that the topology of a space $X$ is coarser than the
discretizable topology of $X_{A}$. If $X$ has an uniform base, then
$X_{A}$ not only has a $G_{\delta}$-diagonal and an uniform base at
non-isolated points, but also has a $\sigma$-point finite base. In
\cite[Theorem 3.1]{GZ} has shown that a space is a discretization of
a metric space if and only if it is a proto-metrizable space having
a $G_{\delta}$-diagonal.

\begin{theorem}
Each discretizable space of a space having an uniform base is an
open compact and at most boundary-one image of a space having an
uniform base.
\end{theorem}

\begin{proof}
Let $X$ be a space having an uniform base. By Lemma 2.4, there is a
point-finite development $\{\mathcal{U}_{m}\}$ for $X$, where
$\mathcal{U}_{m+1}$ refines $\mathcal{U}_{m}$ for each $m\in
\mathbb{N}$. For each $A\subset X$, put

\hspace{0.5cm}$H=(X\times \{0\})\cup (A\times \mathbb{N})$;

\hspace{0.5cm}$V(x,m)=\{x\}\times (\{0\}\cup \{n\in\mathbb N: n\geq
m\}),x\in X, m\in \mathbb{N}$;

\hspace{0.5cm}$W(J,m)=((J\cap (X-A))\times \{0\})$

\hspace{2.5cm}$\cup ((J\cap A)\times \{n\in\mathbb N: n\geq
m\}),J\subset X, m\in \mathbb{N}.$
\\Endow $H$ with a base consisting of the following elements:\

\hspace{1cm}$V(x,m), \forall x\in A, m\in \mathbb{N}$;

\hspace{1cm}$W(J,m), \forall$ open subset $J\subset X, m\in
\mathbb{N}$;

\hspace{1cm}$\{x\}, x\in A\times \mathbb{N}$.
\\Then $H$ is a $T_{2}$-space.

For any $m\in \mathbb{N}$, let

\hspace{1cm}$\mathcal{P}_{m}=\{V(x,m):x\in A\}\cup\{W(U,m):U\in
\mathcal{U}_{m}\}$

\hspace{3.5cm}$\cup\{\{h\}:h\in A\times \{1,2,\cdots,m-1\}\}.$
\\Then $\{\mathcal{P}_{m}\}_{m\geq 2}$ is a point-finite
development for $H$. Hence $H$ has an uniform base.

Let $\pi_{1}|_{H}:H\rightarrow X_A$ be the projective map. It is
easy to see that $\pi_{1}|_{H}$ is an open compact and at most
boundary-one map.
\end{proof}

Hence, each discretizable space of a space having an uniform base is
in MOBI \cite{Be}.

C. Liu \cite{LC} gave some characterizations of quotient (resp.
pseudo-open) boundary-compact images of metric spaces. The following
are further results.

\begin{theorem}
The following are equivalent for a space $X$:
\begin{enumerate}
\item $X$ is first-countable;

\item $X$ is an image of a metric space under a pseudo-open,\
at most boundary-one (resp. boundary-compact) map;

\item $X$ is an image of a metric space under a bi-quotient,\
at most boundary-one (resp. boundary-compact) map.
\end{enumerate}
\end{theorem}

\begin{proof}
$(1)\Leftrightarrow (2)$ was proved in \cite[Corollary 2.1]{LC}, and
$(2)\Leftrightarrow (3)$ is true by Lemma 2.5.
\end{proof}

\begin{theorem}
The following are equivalent for a space $X$:
\begin{enumerate}
\item $X$ has a point-countable base;

\item $X$ is a countably bi-quotient,\ $s$-image of a metric space;

\item $X$ is a pseudo-open,\ boundary-compact and $s$-image of a metric
space;

\item $X$ is a bi-quotient,\ at most boundary-one and countable-to-one
image of a metric space.
\end{enumerate}
\end{theorem}

\begin{proof}
C. Liu proved that a space has a point-countable base if and only if
it is a pseudo-open,\ at most boundary-one and countable-to-one
image of a metric space in \cite{LC}. Thus $(1)\Leftrightarrow (4)$
by Lemma 2.5.\ $(4)\Rightarrow (3)$ is trivial. $(3)\Rightarrow (2)$
by Lemma 2.5, and $(2)\Leftrightarrow (1)$ by \cite{Mi}.
\end{proof}

\vskip 0.5cm
\section{Examples}
In this section some examples are given, which show certain
relations among boundary-compact images of metric spaces and
generalized metric spaces.

\begin{example}
Let $X$ be the closed unit interval $\mathbb I=[0, 1]$ and $B$ a
Bernstein subset of $X$.\ In other words, $B$ is an uncountable set
which contains no uncountable closed subset of $X$.\ The
discretizable space $X_{B}$ is called {\it Michael line} \cite{ME}.\

Let $X^*$ be a copy of $X_{B}$, and $f:X_B\to X^*$ a homeomorphism.
Put $Z=X_{B}\bigoplus X^*$, and let $Y$ a quotient space obtained
from $Z$ by identifying $\{x, f(x)\}$ to a point for each $x\in
X_B\setminus B$. Then

(1)\ $X_{B}$ is a discretizable space of the metric space $\mathbb
I$, so it is a proto-metrizable space, and an open compact, at most
boundary-one image of a space with an uniform base by Theorem 3.7.

(2)\ $X_{B}$ is not a BCO space, hence it is not an open compact
image of  a metric space;

(3)\ $Y$ is an open boundary-compact, $s$-image of a metric space;

(4)\ $Y$ has not any $G_{\delta}$-diagonal by \cite[Example 1]{Po}.
\end{example}

It is obvious that $X_B$ is a paracompact space which is a
discretizable space of the metric space $\mathbb I$. If $X_B$ is
BCO, it is a developable space, then $B$ is an $F_{\sigma}$-set in
$X_{B}$, a contradiction. Thus $X_B$ is not BCO.

It is easy to check that $Y$ has a point-countable base which is
uniform at non-isolated points. Hence $Y$ is an open
boundary-compact, $s$-image of a metric space by Corollary 3.2.

\begin{example}
Let $\psi(D)$ be {\it Isbell-Mr\'{o}wka} space \cite{MS}, here
$|D|\geq\aleph_0$. Then

(1)\ $\psi(D)$ is an open,  boundary-compact image of a metric
space;

(2)\ $\psi(D)$ is not a meta-Lindel\"{o}f space;

(3)\ $\psi (D)$ is a developable space if $|D|=\aleph_0$;

(4)\ $\psi(D)$ is not a perfect space if $|D|\geq\bf c$.
\end{example}

A collection $\mathcal C$ of subsets of an infinite set $D$ is said
to be {\it almost disjoint} if $A\cap B$ is finite whenever
$A\not=B\in\mathcal C$. Let $\mathcal A$ be an almost disjoint
collection of countably infinite subsets of $D$ and maximal with
respect to the properties. Then $|\mathcal A|\geq |D|^+$ \cite{Ku}.
Isbell-Mr\'{o}wka space $\psi(D)$ is the set $\mathcal A\cup D$
endowed with a topology as follows: The points of $D$ are isolated.
Basic neighborhoods of a point $A\in\mathcal A$ are the sets of the
form $\{A\}\cup (A-F)$ where $F$ is a finite subset of $D$.

Let $X=\psi(D), \mathcal{A}=\{A_{\alpha}\}_{\alpha\in \Lambda}$ and
each $A_{\alpha}=\{x(\alpha, n):n\in \mathbb{N}\}$.\ For each
$n\in\mathbb N$, put
$$\mathcal{B}_{n}=\{\{A_{\alpha}\}\cup \{x(\alpha,
m):m\geq n\}:\alpha\in \Lambda\}\cup \{\{x\}:x\in D\}.$$

It is easy to see that $\{\mathcal{B}_{n}\}$ is a point-finite
development for $X$.\ Thus $X$ is the open, boundary-compact image
of a metric space by Theorem 3.1. Since an open cover
$\{\{A_\alpha\}\cup D\}_{\alpha\in\Lambda}$ of $X$ has not any
point-countable open refinement, $X$ is not a meta-Lindel\"{o}f
space. Thus $X$ is not an open $s$-image of a metric space, and $X$
is not a discretizable space of a space with an uniform base by
Theorem 3.7.

If $D$ is countable, it is obvious that $\psi (D)$ is a developable
space. Hence $\psi (D)$ has a $G_{\delta}$-diagonal, but $\psi (D)$
has not any point-countable base because  $\psi(D)$ is not a
meta-Lindel\"{o}f space.

If $|D|\geq\bf c$, $\psi (D)$ is not a developable space \cite{C1},
thus $\psi(D)$ is not perfect by Corollary 3.5.

\begin{example}
There is a space $X$ such that

(1)\ $X$ has a sharp base;

(2)\ $X$ has not an uniform base at non-isolated points;

(3)\ $X$ is an open compact and countable-to-one image of a space
with an uniform base.
\end{example}

A space $X$ having the properties is constructed in \cite[Example
5.1]{AAC}, where it is shown that $X$ has a non-developable space
with a sharp base.\ Since $X$ has not any isolated point,\ it is not
an open,\ boundary-compact image of a metric space and has not an
uniform base at non-isolated points by Theorem 3.1. J. Chaber in
\cite[Example 4.5]{C2} was proved that $X$ is an open compact and
countable-to-one image of a space with an uniform base.

\begin{example}
There is a bi-quotient,\ at most boundary-one image $X$ of a metric
space such that $X$ is neither a pseudo-open $s$-image of a metric
space, nor an open,\ boundary-compact image of a metric space.
\end{example}

Let $X=\mathbb{R}^{2}$ endowed with the butterfly topology\cite{ML}.
It is easy to see that $X$ is a first-countable,\ paracompact space
without any isolated point.\ Since $X$ is a first-countable space,\
then $X$ is a bi-quotient,\ at most boundary-one image of a metric
space by Theorem 3.8.\ Since $X$ has not a point-countable base
\cite[Example 1.8.3]{Ls}, $X$ is not a countably bi-quotient
$s$-image of a metric space by Theorem 3.9. Because each pseudo-open
map from a space onto a first-countable space is countably
bi-quotient \cite{Mi}, $X$ is not a pseudo-open $s$-image of a
metric space.\ If $X$ is an open,\ boundary-compact image of a
metric space,\ $X$ is an open compact image of a metric space for
$X$ does not contain any isolated point.\ So $X$ is a developable
space by Lemma 2.4.\ Thus $X$ is a metric space,\ a contradiction.\

\begin{example}
There a proto-metrizable space without any uniform base at
non-isolated points.
\end{example}

G. Gruenhage in \cite[p. 363]{Gr1} constructed a proto-metrizable
$X$ which is not a $\gamma$-space. Hence $X$ has not any $\sigma$-Q
base by \cite[Proposition 1.7.10]{Ls}, and it has not any uniform
base at non-isolated points by Theorem 3.4.

\begin{example}
There a space such that it is an open compact image of a metric
space, which is not any open, at most boundary-one image of a metric
space.
\end{example}

Y. Tanaka in \cite[Example 3.1]{Ta} constructed a non-regular
$T_{2}$-space $X$ which is an open, at most two-to-one image of a
metric space. Since $X$ has not any isolated point, it is not an
open, at most boundary-one image of a metric space. Otherwise, $X$
is an image of a metric space under an open and bijective map, then
$X$ is homeomorphic to a metric space, a contradiction.

\vskip 0.5cm
\section{Questions}
Some questions are posed in the final.

\begin{question}
Let a space $X$ have a point-countable base. If $X$ has an uniform
base at non-isolated points, is $X$ an open, boundary-compact,
$s$-image of a metric space?
\end{question}

\begin{question}
Is an open and boundary-compact $s$-image of a metric space an open,
boundary-compact and countable-to-one image of a metric space?
\end{question}

\begin{question}
How to character a discretizable space of a space with an uniform
base by a certain image of a metric space? For example, whether the
open compact and at most boundary-one image of a space with an
uniform base is a discretizable space of a space with an uniform
base?
\end{question}

\begin{question}
How to character a space which is an open, at most boundary-one,
$s$-image of a metric space?
\end{question}

{\bf Acknowledgements}. The authors would like to thank the referee
for his valuable suggestions.

\end{document}